\documentclass[twoside,11pt]{article}

\usepackage{blindtext}

\usepackage[preprint]{jmlr2e}

\usepackage[utf8]{inputenc}
\usepackage[T1]{fontenc}
\usepackage{lmodern}
\usepackage{amsmath}
\usepackage{paralist}
\usepackage{float}
\usepackage{accents}
\usepackage{color}
\usepackage{bm}
\usepackage{lscape}
\usepackage{subfig}
\usepackage{longtable}
\usepackage{multibib}
\usepackage{booktabs}
\newcites{suppl}{References} 
\usepackage{tikz}
\usetikzlibrary{positioning}
\usepackage{dsfont}
\usepackage{pgfplotstable}
\usepackage{hyperref}
\usepackage{lastpage}
\usepackage{multirow}
\usepackage{setspace}

\newcommand{\eps}{{\varepsilon}}
\renewcommand{\phi}{\varphi}
\newcommand{\R}{\mathbb{R}}
\newcommand{\Z}{\mathbb{Z}}
\newcommand{\N}{\mathbb{N}}
\newcommand{\pr}{\mathbb{P}}       
\newcommand{\ex}{\mathbb{E}}

\newcommand{\Ec}{\mathcal{E}}
\newcommand{\Oc}{\mathcal{O}}
\newcommand{\diff}{{\,\mathrm{d}}}

\newcommand{\C}{V}   

\DeclareMathOperator*{\argmin}{argmin}

\newtheorem{assumption}[theorem]{Assumption}

\jmlrheading{24}{2024}{1-\pageref{LastPage}}{7/24}{-}{}{Florian Heinrichs}

\ShortHeadings{Local Linear Regression for Time Series in Banach Spaces}{Heinrichs}
\firstpageno{1}

\begin{document}
	
\title{A Note on Local Linear Regression for Time Series in Banach Spaces}

\author{\name Florian Heinrichs \email f.heinrichs@fh-aachen.de \\
	\addr FH Aachen\\
	Heinrich-Mußmann-Straße 1\\
	52428 Jülich, Germany}

\editor{-}

\maketitle

\begin{abstract}%
	This note extends local linear regression to Banach space-valued time series for estimating smoothly varying means and their derivatives in non-stationary data. The asymptotic properties of both the standard and bias-reduced Jackknife estimators are analyzed under mild moment conditions, establishing their convergence rates. Simulation studies assess the finite sample performance of these estimators and compare them with the Nadaraya-Watson estimator. Additionally, the proposed methods are applied to smooth EEG recordings for reconstructing eye movements and to video analysis for detecting pedestrians and abandoned objects.
\end{abstract}


\maketitle 

\section{Introduction} \label{sec:intro}

Many physical processes are smooth in time, space or more abstract arguments, such as the frequency in a spectral density. Since smooth processes occur in many areas, there are different approaches to their analysis. In statistics, on the one hand, functional data analysis deals with smooth data, and on the other hand, the framework of local stationarity allows an analysis of time-dependent data whose behavior changes smoothly. A natural combination of functional data and time series analysis leads to functional time series, in which each observation at a point in time corresponds to a (continuous) function. One example of such data are videos, with each observation being a function of the $x$- and $y$-coordinate. Another example of functional time series arises in meteorology, when the temperature as a function of space is studied over time.
 
Depending on the type of functional observation, the time series takes values in a different space. We are often interested in $C([0, 1])$, the space of continuous functions, or $L^2([0, 1])$, the space of square integrable functions \citep{bucher2020}. For convenience, we rescale the domain of the functional observations to the unit interval $[0, 1]$. Sometimes we are also interested in more general spaces, e.\,g., the space of $k$-times continuously differentiable functions $C^k([0, 1])$, or spaces of $d$-dimensional functions, such as $\big(C([0, 1])\big)^d$ or $\big(L^2([0, 1])\big)^d$. All these spaces are Banach spaces, so in the following we consider time series in a Banach space $\C$ with norm $\|\cdot\|$. Note that this setting also includes univariate and multivariate time series with $\C=\R$ or $\C=\R^n$.
 
For a time series $X_1, \dots, X_n\in \C$, consider the location scale model
\begin{equation}\label{eq:model}
	X_{i}=\mu\big(\tfrac{i}{n}\big) + \eps_{i}\quad i=1,\dots, n,
\end{equation}
where $\mu:[0, 1]\to\C$ denotes the unknown mean operator and $(\eps_{i})_{i\in\Z}$ is a centered error process in $\C$. We are interested in estimating the mean operator $\mu$ and, for sufficiently smooth $\mu$, also its derivative(s). Estimating the mean operator is often only the first step in a statistical analysis of functional time series. With the estimator, we can then test the time series for outliers, structural breaks or smooth changes \citep{bucher2021, heinrichs2025}. We can also use the mean value operator and its derivative(s) to predict the future course of the time series. In the following, we will generalize the commonly used local linear estimator to Banach space-valued time series and study its convergence. 

Similar regression problems of the form $\nu:\C\to \R$ have been studied extensively in the literature, and are referred to as ``scalar-on-function'' regression. Local linear regression was used for the estimation of $\nu$ under the assumption of i.i.d. observations \citep{benhenni2007, baillo2009, boj2010, barrientos2010, berlinet2011, ferraty2022, bhadra2024}. In the i.i.d. setting, local linear estimation was also used for the estimation of the conditional density and distribution function of a scalar random variable $Y$ given a functional observation $X$ \citep{demongeot2013, demongeot2014}. More recently, local linear regression has been studied for stationary time series \cite{demongeot2011, xiong2018, bouanani2024}, and was extended even to locally stationary functional time series \cite{kurisu2022}.

In a different line of work, \cite{benhenni2014} used local polynomial estimation to estimate $m:[0,1]\to \R$ under the model $Y_i(x_j) = m(x_j) + \eps_i(x_j)$, for grid points $(x_j)_{j=1}^N$ and i.i.d. observations of $Y_i$, for $i=1, \dots, n$. This is a special case of \eqref{eq:model}, where $\mu(i/n)\equiv m$ is constant (in time), and $m\in \C$. 

The problem of estimating $\mu$ in \eqref{eq:model} has been considered less in the literature. To the best of our knowledge, only \cite{kurisu2021} studied the Nadaraya-Watson estimator for locally stationary functional time series.

In the following, we extend the local linear estimator to Banach space-valued time series. In contrast to the Nadaraya-Watson estimator, this allows us to directly estimate $\mu$ and its derivative. We derive the convergence rate of the local linear estimator and its bias-reduced Jackknife version in Section \ref{sec:method}. Further, we study the finite sample properties of the estimators and compare them to the Nadaraya-Watson estimator in Section \ref{sec:sim_data}. In addition, we apply the estimators to smooth EEG recordings and reconstruct eye movements from the smoothed EEG data in Section \ref{sec:real_data}. In the same section, we use the smoothing methods to recognize pedestrians and abandoned objects in videos. The proof of the main results from Section \ref{sec:method} is deferred to Section \ref{sec:proof}.

\section{Methodology} \label{sec:method}

In the following, we introduce some notation and subsequently specify necessary regularity conditions. Let $V, W$ denote Banach spaces and $f:V \to W$ an operator between them. Then, the Gateaux differential of $f$ at $v\in V$ in direction $\psi\in V$ is defined as
\[ \diff f(v; \psi)= \lim_{\tau\to 0}\frac{f(v+\tau \psi)-f(v)}{\tau}. \]
If the Gateaux differential exists for any direction $\psi\in V$, $f$ is Gateaux differentiable. Note that the Gateaux differential does not need to be linear. A more restrictive notion for operators is Fréchet differentiability. The operator $f$ is Fréchet differentiable if a bounded linear operator $A:V\to W$ exists such that
\[ \lim_{\|h\|_V\to 0}\frac{\|f(v+h)-f(v)-Ah\|_W}{\|h\|_V} = 0. \]
In this case, $A$ is the Fréchet derivative of $f$ in $v$ and denoted by $Df(v)$.

Throughout, we assume $\mu$ to be Fréchet differentiable, which is further specified in Assumption \ref{assump:mu}. In order to estimate $\mu$, we use local linear regression. More specifically, define the local linear estimator of $\mu$ and its derivative as
\begin{equation}\label{eq:def}
	\Big(\hat{\mu}_{h_n}(t), \widehat{D\mu}_{h_n}(t)\Big) 
	= \argmin_{b_0, b_1\in \C} \bigg\| \sum_{i=1}^{n} \Big(X_{i} - b_0 - b_1 \big(\tfrac{i}{n}-t\big) \Big)^2 K\big(\tfrac{i-nt}{nh_n}\big) \bigg\|.
\end{equation} 

for some kernel $K$ and bandwidth $h_n\searrow 0$. To account for the bias of this estimator, we employ the Jackknife bias reduction technique proposed by \cite{schucany1977} and define
\[ \tilde{\mu}_n(t) = 2 \hat{\mu}_{h_n/\sqrt{2}}(t) - \hat{\mu}_{h_n}(t). \]
Similar, we reduce the bias of $\widehat{D\mu}_{h_n}(t)$ by defining
\[ \widetilde{D\mu}_n(t) = \tfrac{\sqrt{2}}{\sqrt{2}-1} \widehat{D\mu}_{h_n/\sqrt{2}}(t) - \tfrac{1}{\sqrt{2}-1} \widehat{D\mu}_{h_n}(t) \]

Note that for $\C=\R$, we obtain the usual local linear estimator, and for $\C=C([0, 1])$, the estimator defined in \eqref{eq:def} equals the pointwise defined local linear estimation, used in the literature \citep{bastian2025}.  

The performance of the estimators $\tilde{\mu}_n$ and $\widetilde{D\mu}_n$ depend on the properties of the kernel $K$, the mean operator $\mu$ and the errors $(\eps_i)_{i\in\Z}$. We now state regularity conditions to establish consistency of the estimators.

\begin{assumption} \label{assump}
	\begin{enumerate}
		\item \label{assump:kern}
		The kernel $K:\R\to\R$ is non-negative, symmetric and supported on $[-1, 1]$. It is twice differentiable, satisfies $\int_{-1}^1K(x)\diff x = 1$ and is Lipschitz continuous in an open set containing $[-1, 1]$.
		\item \label{assump:mu}
		The operator $\mu$ is twice Fréchet differentiable with Lipschitz continuous second derivative. 
		\item \label{assump:errors}
		The errors $(\eps_i)_{i\in\Z}$ have (jointly) bounded $p$-th moment, i.\,e., 
		\[ \sup_{i\in\Z} \ex[\|\eps_i\|^p ] \le C \]
		for some $p \ge 1$ and a constant $C\in\R$.
	\end{enumerate}
\end{assumption}

\begin{remark}
	Assumption \ref{assump} is rather mild. Note that the first part is essentially a design choice for the local linear estimator. The second condition ensures that $\mu$ and its derivative are sufficiently smooth to estimate them reasonably well and is common in the literature\citep{bucher2021, heinrichs2021}. The last part is satisfied for most distributions of interest. Note that the more moments exist, the better the estimators will be.
\end{remark}

Based on this assumption, we obtain the following uniform approximation of the Jackknife estimators $\tilde{\mu}_n$ and $\widetilde{D\mu}_n$.

\begin{theorem}\label{thm:lle}
	Let $h_n\searrow 0$ as $n\to\infty$ and define $K^*(x) = 2\sqrt{2}K(\sqrt{2}x)-K(x)$ and $\kappa_2 = \int_{-1}^1 x^2 K(x)\diff x$. Under Assumption \ref{assump}, it holds
	\[ \sup_{t\in [h_n, 1-h_n]} \bigg\| \tilde{\mu}_n(t) - \mu(t) - \frac{1}{nh_n} \sum_{i=1}^{n} \eps_i K^*\big(\tfrac{i-nt}{nh_n}\big)\bigg\| = \Oc(h_n^3) + \Oc_\pr\Big(\tfrac{1}{nh_n^{1+1/p}}\Big) \]
	and
	\[ \sup_{t\in [h_n, 1-h_n]} \bigg\| \widetilde{D\mu}_n(t) - D\mu(t) - \frac{1}{\kappa_2(\sqrt{2}-1)}\frac{1}{nh_n^2} \sum_{i=1}^{n} \big(\tfrac{i-nt}{nh_n}\big) \eps_i K^*\big(\tfrac{i-nt}{nh_n}\big)\bigg\| = \Oc(h_n^2) + \Oc_\pr\Big(\tfrac{1}{nh_n^{2+1/p}}\Big) \]
\end{theorem}

\section{Empirical Results} \label{sec:empirical}

In the following, we study the finite sample properties of the local linear estimator and its Jackknife version and compare it to the Nadaraya-Watson estimator, defined as 
\[ \hat{\mu}_{NW}(t) = \frac{\sum_{i=1}^n X_i K\big(\tfrac{i-nt}{nh_n}\big)}{\sum_{i=1}^n K\big(\tfrac{i-nt}{nh_n}\big)}, \]
for $t\in[0, 1]$ and $K$ and $h_n$ as before \citep{fan1996}. In the remainder of this section, we  consider the equidistant decomposition $\{\tfrac{0}{n}, \tfrac{1}{n}, \dots, \tfrac{n}{n}\}$ of the interval $[0, 1]$. As the Nadaraya-Watson estimator does not yield an estimate of the derivative $D\mu(t)$, we estimate it by calculating 
\[\widehat{D\mu}_{NW}\big(\tfrac{i}{n}\big) = \frac{\hat{\mu}_{NW}\big(\tfrac{i+1}{n}\big) - \hat{\mu}_{NW}\big(\tfrac{i-1}{n}\big)}{\tfrac{2}{n}}, \]
for $i\in \{1, \dots, n-1\}$, $\widehat{D\mu}_{NW}(0) = n \big(\hat{\mu}_{NW}(\tfrac{1}{n}) - \hat{\mu}_{NW}(0)\big)$ and $\widehat{D\mu}_{NW}(1) = n \big(\hat{\mu}_{NW}(1) - \hat{\mu}_{NW}(\tfrac{n-1}{n})\big)$.

In Section \ref{sec:sim_data}, we generate synthetic data to compare the estimators based on the real quantities $\mu$ and $D\mu$. In Section \ref{sec:real_data}, we showcase the estimators in real applications, where smoothing the data is only the first step of a statistical analysis. The code for the subsequent experiments can be found in the GitHub repository: \url{https://github.com/FlorianHeinrichs/banach_space_llr/}

\subsection{Monte Carlo Simulation Study} \label{sec:sim_data}

Given the model $X_i=\mu\big(\tfrac{i}{n}\big) + \eps_i $, as introduced in \eqref{eq:model}, the estimators are compared for different choices of $\mu$ and $\eps$. For the mean operator $\mu$, the choices
\begin{figure}
	\includegraphics[width=\textwidth]{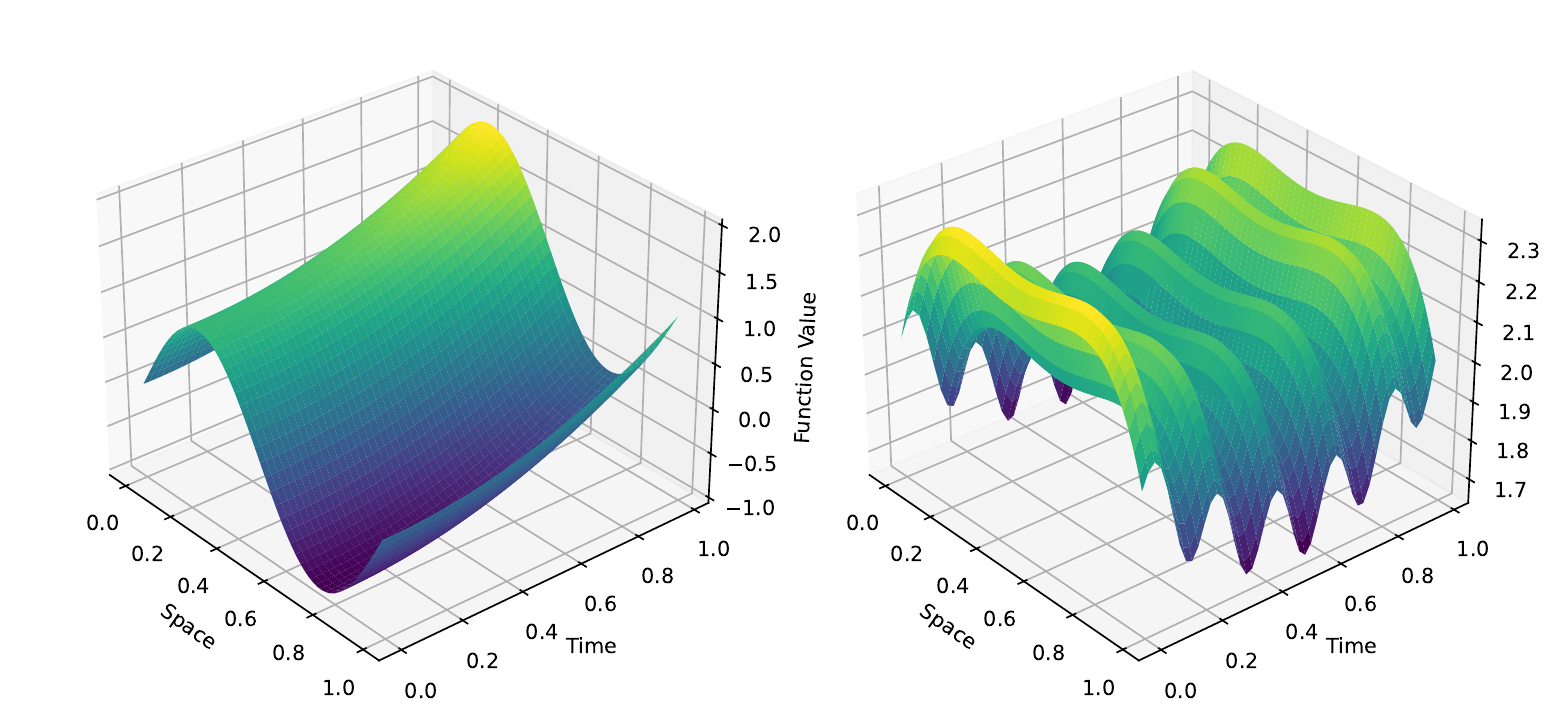}
	\caption{Different choices of the mean operators $\mu$. Left: $\mu_1$. Right: $\mu_2$.}
	\label{fig:mu}
\end{figure}
\begin{align*}
	\mu_1(t) & = \sin(2 \pi \cdot) + t^2 \\
	\mu_2(t) & = \phi(\cdot) + \big(t - \tfrac{1}{2}\big)^2 + \tfrac{1}{10}\sin(10 \pi t) + \tfrac{3}{4},
\end{align*}
where $\phi(x) = -8x^4 + 16x^3 - 11x^2 + 3x+1$, are studied, as displayed in Figure \ref{fig:mu}. The different mean operators were selected to evolve monotonously and non-monotonously in time, and non-monotonously in space. Note that $\mu_i \in C([0, 1]^2)$, for $i=1, 2$, and in particular $\mu_i(t) \in C([0, 1])$, for each $t\in[0, 1]$. The error processes were adapted from \cite{bucher2023}. More specifically, let $(\eta_i^{(1)})_{i\in\N}$ denote an i.i.d. sequence of Brownian motions and $(\eta_i^{(2)})_{i\in\N}$ an i.i.d. sequence of Brownian bridges. We generated the following error processes
%
\[
\begin{array}{llll}
	\text{(BM)} & \eps_i = \eta_i^{(1)}, 
	& \text{(BB)} & \eps_i = \eta_i^{(2)}, \\
	\text{(FAR-BM)} & \eps_i = \rho(\eps_{i-1}) + \eta_i^{(1)}, 
	& \text{(FAR-BB)} & \eps_i = \rho(\eps_{i-1}) + \eta_i^{(2)}, \\
	\text{(tvBM)} & \eps_i = \sigma(\tfrac{i}{n})\eta_i^{(1)}, & \\
	\text{(tvFAR1)} & \eps_i = \rho(\eps_{i-1}) + \sigma(\tfrac{i}{n})\eta_i^{(1)},
	& \text{(tvFAR2)} & \eps_i = \sigma(\tfrac{i}{n}) \rho(\eps_{i-1}) + \eta_i^{(1)},\\
\end{array}
\]

where $\sigma(x) = x + \tfrac{1}{2}$ and $\rho$ denotes the integral operator $\rho(f) = \int_0^1 K(\cdot, x) f(x) \diff x$, for $K(y, x) = 0.3 \sqrt{6} \min\{x, y\}$. The error processes were selected to cover  independent and dependent, stationary and non-stationary time series. Trajectories of error processes (FAR-BM), (tvBM) and (tvFAR1) are displayed in Figure \ref{fig:epsilon}.

\begin{figure}
	\includegraphics[width=\textwidth]{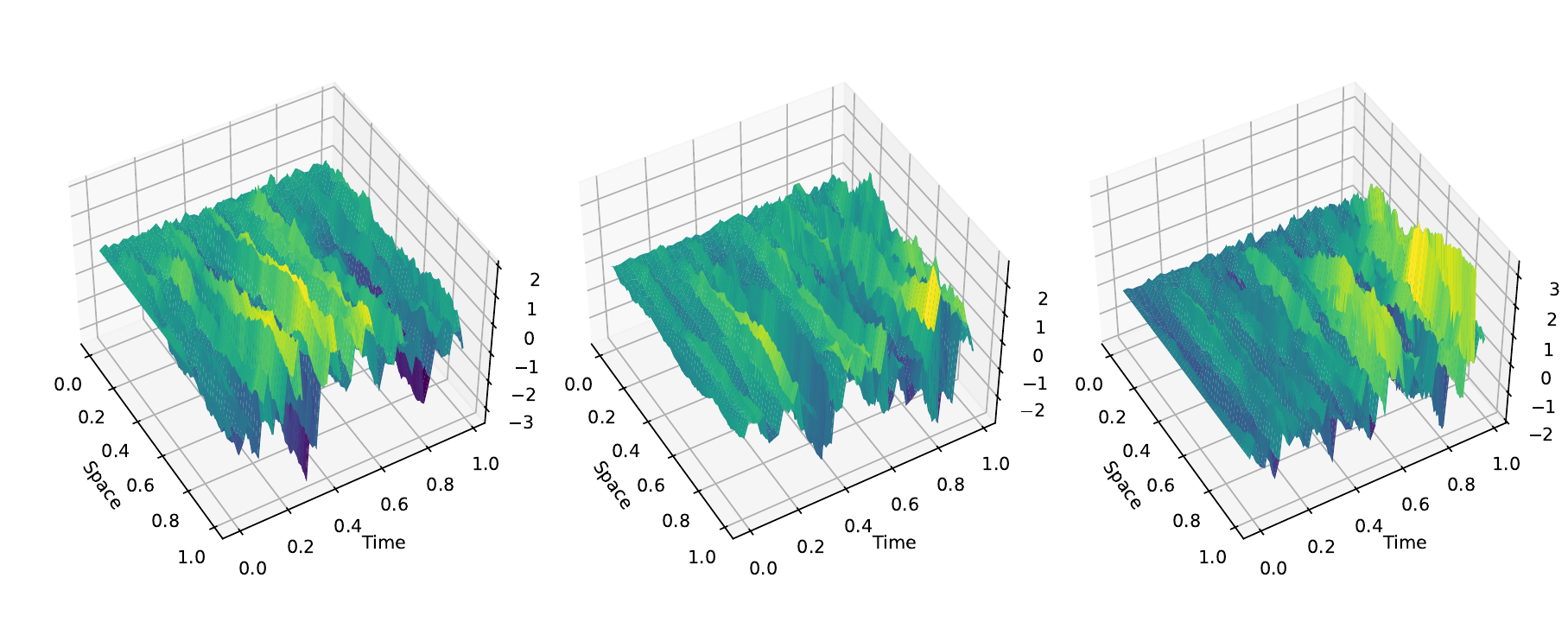}
	\caption{Different choices of the error process $\eps$. Left: (FAR-BM). Center: (tvBM) Right: (tvFAR1).}
	\label{fig:epsilon}
\end{figure}

For each combination of mean operator and error process, we generated $1000$ trajectories of length $n= 50, 100, 200$ and $500$. Further, we chose the spatial resolution $m=100$, so that we used the grid points $\tfrac{i}{m-1}$, for $i=0 ,\dots, m-1$. 
For all estimators, we used the quartic kernel $K(x) = \tfrac{15}{16}(1-x^2)^2$ and selected the bandwidth $h_n$ between $\tfrac{1}{n}$ and $\tfrac{1}{\sqrt{n}}$ via $k$-fold cross-validation with $k=5$ folds, minimizing the mean-squared error (MSE) between predictions and observations in the validation data. 

For the comparison, we evaluated the MSE of the estimators with respect to the true mean operator $\mu$ and its derivative $D\mu$, defined as
\[ MSE(\hat{\mu}) = \frac{1}{n} \sum_{i=1}^{n} \|\hat{\mu}(\tfrac{i}{n}) - \mu(\tfrac{i}{n})\|_2^2 = \frac{1}{n} \sum_{i=1}^{n} \frac{1}{m} \sum_{j=0}^{m-1} \Big(\hat{\mu}(\tfrac{i}{n}, \tfrac{j}{m-1}) - \mu(\tfrac{i}{n}, \tfrac{j}{m-1})\Big)^2. \]
The experiment's results are shown in Table \ref{tab:results1}. Additionally, we calculated the mean absolute error (MAE), defined with respect to the $L^1$-norm as 
\[ MAE(\hat{\mu}) = \frac{1}{n} \sum_{i=1}^{n} \|\hat{\mu}(\tfrac{i}{n}) - \mu(\tfrac{i}{n})\|_1 = \frac{1}{n} \sum_{i=1}^{n} \frac{1}{m} \sum_{j=0}^{m-1} \big|\hat{\mu}(\tfrac{i}{n}, \tfrac{j}{m-1}) - \mu(\tfrac{i}{n}, \tfrac{j}{m-1})\big| \]
and the processing time of each estimator. The respective results are listed in Tables \ref{tab:times} and \ref{tab:results2} of Appendix \ref{app:appendix_empirical_results}.

Generally, for a given error process $\eps$, the MSE is very similar for both mean operators $\mu_1$ and $\mu_2$. Also, the MSEs for the derivative are of a similar order. 
Overall, the Jackknife estimator has higher MSEs compared to the local linear and the Nadaraya Watson estimators. The MSEs are substantially lower for the estimates of $\mu$ compared to the estimates of $D\mu$, which may have two possible causes. First, estimating the derivative $D\mu$ is generally known to be more difficult and the asymptotic rates decrease slower. Second, the bandwidth $h_n$ was selected by minimizing the MSE of $\hat{\mu}$, rather than that of $\widehat{D\mu}$. This also explains why the MSEs increase for $D\mu$, as $n$ grows. Generally, the estimation of derivatives requires larger bandwidths.
For the estimation of $\mu$, the local linear estimator $\hat{\mu}_{h_n}$ has generally the lowest MSE, while for the estimation of $D\mu$, the Nadaraya-Watson estimator seems favorable. In the latter case, however, the situation is not as clear, and $\widehat{D\mu}_{h_n}$ has a lower MSE than $\widehat{D\mu}_{NW}$ for some error processes. Finally, the MAE reflects the results of the MSE.

The processing time of all estimators seems to grow approximately linearly with $n$. As expected, there are no substantial differences between the variants of $\mu$ and $\eps$. The processing time of the Jackknife estimator is about twice as long as that of the local linear estimator, which is clear from the definition of the former. The Nadaraya-Watson estimator takes roughly half the time of the local linear estimator. Again, this is expected, as the Nadaraya-Watson estimator can be expressed as $\tfrac{R_0(t)}{S_0(t)}$, with the notation from \eqref{eq:s_r}, while the local linear estimator equals $\tfrac{S_2(t) R_0(t) - S_1(t)R_1(t)}{S_2(t)S_0(t)-S_1^2(t)}$. Since $S_1(t)$ converges to $0$, uniformly in $t$, this also explains why both estimators take similar values.

\setlength{\tabcolsep}{3pt}

\begin{table}\small
\caption{MSE of the compared estimators for different choices of $\mu$ and $\eps$.}
\label{tab:results1}
\begin{tabular}{r|rrrrrr|rrrrrr}
	& \multicolumn{6}{c|}{$\mu_1$} & \multicolumn{6}{c}{$\mu_2$} \\
	$n$ & $\tilde{\mu}_n$ & $\hat{\mu}_{h_n}$ & $\hat{\mu}_{NW}$ & $\widetilde{D\mu}_n$ & $\widehat{D\mu}_{h_n}$ & $\widehat{D\mu}_{NW}$ & $\tilde{\mu}_n$ & $\hat{\mu}_{h_n}$ & $\hat{\mu}_{NW}$ & $\widetilde{D\mu}_n$ & $\widehat{D\mu}_{h_n}$ & $\widehat{D\mu}_{NW}$ \\
	\midrule
	\addlinespace[.2cm]
	\multicolumn{13}{l}{\quad\textit{Panel A: (BM)}} \\
	50 & 0.23 & \textbf{0.07} &\textbf{ 0.07} & 1489.95 & 21.96 & \textbf{14.25} & 0.15 & \textbf{0.07} & \textbf{0.07} & 407.80 & 20.75 & \textbf{17.71} \\
	100 & 0.08 & \textbf{0.04} & 0.05 & 278.17 & 26.51 & \textbf{20.08} & 0.08 &\textbf{ 0.04} & 0.05 & 267.60 & 25.56 & \textbf{19.84} \\
	200 & 0.06 & \textbf{0.03} & \textbf{0.03} & 520.68 & 33.60 & \textbf{23.69} & 0.06 & \textbf{0.03} & \textbf{0.03} & 518.53 & 34.46 & \textbf{26.90} \\
	500 & 0.04 & \textbf{0.02} & \textbf{0.02} & 729.17 & 49.02 & \textbf{35.12} & 0.04 & \textbf{0.02} & \textbf{0.02} & 725.79 & 50.05 & \textbf{35.65} \\
	\addlinespace[.2cm]
	\multicolumn{13}{l}{\quad\textit{Panel B: (BB)}} \\
	50 & 0.06 & \textbf{0.02} & \textbf{0.02} & 242.42 & 5.79 & \textbf{3.79} & 0.05 & \textbf{0.02} & \textbf{0.02} & 128.68 & 8.66 & \textbf{7.35} \\
	100 & 0.03 & \textbf{0.01} & \textbf{0.01} & 88.83 & 6.96 & \textbf{4.99} & 0.03 & \textbf{0.02} & \textbf{0.02} & 88.98 & 8.87 & \textbf{7.12} \\
	200 & 0.02 & \textbf{0.01} & \textbf{0.01} & 169.14 & 10.79 & \textbf{7.35} & 0.02 & \textbf{0.01} & \textbf{0.01} & 168.27 & 10.60 & \textbf{7.77} \\
	500 & \textbf{0.01} & \textbf{0.01} & \textbf{0.01} & 239.72 & 15.90 & \textbf{11.27} & \textbf{0.01} & \textbf{0.01} & \textbf{0.01} & 236.77 & 15.99 & \textbf{11.27} \\
	\addlinespace[.2cm]
	\multicolumn{13}{l}{\quad\textit{Panel C: (FAR-BM)}} \\
	50 & 0.23 & \textbf{0.12} & 0.13 & 491.29 & 57.50 & \textbf{49.49} & 0.23 & \textbf{0.12} & 0.13 & 505.46 & 44.18 & \textbf{42.76} \\
	100 & 0.15 & \textbf{0.09} & 0.11 & 571.23 & \textbf{99.19} & 114.85 & 0.15 & \textbf{0.10} & 0.11 & 531.52 & 134.38 & \textbf{128.65} \\
	200 & 0.12 & \textbf{0.08} & 0.09 & 1282.39 & \textbf{311.93} & 316.70 & 0.12 & \textbf{0.09} & 0.10 & 1195.59 & 386.11 & \textbf{371.58} \\
	500 & 0.08 & \textbf{0.07} & \textbf{0.07} & 2853.84 & 1755.15 & \textbf{1216.03} & 0.09 & \textbf{0.07} & 0.08 & 3463.97 & 1883.35 & \textbf{1460.87} \\
	\addlinespace[.2cm]
	\multicolumn{13}{l}{\quad\textit{Panel D: (FAR-BB)}} \\
	50 & 0.06 & \textbf{0.03} & \textbf{0.03} & 146.52 & 8.60 & \textbf{7.31} & 0.06 & \textbf{0.03} & \textbf{0.03} & 146.06 & \textbf{9.82} & 11.18 \\
	100 & 0.04 & \textbf{0.02} & \textbf{0.02} & 117.84 & 18.15 & \textbf{10.70} & 0.04 & \textbf{0.02} & 0.03 & 116.00 & 17.58 & \textbf{17.51} \\
	200 & 0.03 & \textbf{0.02} & \textbf{0.02} & 244.74 & 24.44 & \textbf{23.75} & 0.03 & \textbf{0.02} & \textbf{0.02} & 235.29 & 30.50 & \textbf{24.52} \\
	500 & 0.02 & \textbf{0.01} & \textbf{0.01} & 384.68 & 42.72 & \textbf{33.77} & 0.02 & \textbf{0.01} & \textbf{0.01} & 395.22 & 46.48 & \textbf{38.91} \\
	\addlinespace[.2cm]
	\multicolumn{13}{l}{\quad\textit{Panel E: (tvBM)}} \\
	50 & 0.16 & \textbf{0.07} & \textbf{0.07} & 440.63 & 34.95 & \textbf{18.31} & 0.17 & \textbf{0.08} & \textbf{0.08} & 470.08 & 30.29 & \textbf{19.57} \\
	100 & 0.09 & \textbf{0.05} & \textbf{0.05} & 338.28 & 47.99 & \textbf{24.56} & 0.09 & \textbf{0.05} & \textbf{0.05} & 298.05 & 33.68 & \textbf{24.07} \\
	200 & 0.07 & \textbf{0.03} & \textbf{0.03} & 617.61 & 41.53 & \textbf{28.67} & 0.07 & \textbf{0.03} & \textbf{0.03} & 581.39 & 44.70 & \textbf{31.30} \\
	500 & 0.04 & \textbf{0.02} & \textbf{0.02} & 814.03 & 58.66 & \textbf{40.34} & 0.04 & \textbf{0.02} & \textbf{0.02} & 800.36 & 59.56 & \textbf{42.43} \\
	\addlinespace[.2cm]
	\multicolumn{13}{l}{\quad\textit{Panel F: (tvFAR1)}} \\
	50 & 0.39 & \textbf{0.14} & 0.16 & 971.13 & 99.12 & \textbf{72.93} & 0.25 & \textbf{0.13} & 0.15 & 539.90 & \textbf{49.35} & 56.09 \\
	100 & 0.16 & \textbf{0.11} & 0.14 & 556.02 & \textbf{170.90} & 186.98 & 0.16 & \textbf{0.13} & \textbf{0.13} & 631.44 & 209.00 & \textbf{151.46} \\
	200 & 0.14 & \textbf{0.09} & 0.11 & 1589.57 & \textbf{326.21} & 457.70 & 0.14 & \textbf{0.09} & 0.10 & 1840.00 & \textbf{322.72} & 402.68 \\
	500 & 0.10 & \textbf{0.08} & 0.09 & 3959.31 & \textbf{1597.93} & 1810.27 & 0.10 & \textbf{0.08} & 0.09 & 3858.18 & 2051.42 & \textbf{1528.49} \\
	\addlinespace[.2cm]
	\multicolumn{13}{l}{\quad\textit{Panel G: (tvFAR2)}} \\
	50 & 0.24 & \textbf{0.13} & 0.14 & 538.32 & 59.93 & \textbf{49.89} & 0.23 & \textbf{0.11} & 0.14 & 486.28 & \textbf{33.25} & 48.05 \\
	100 & 0.16 & \textbf{0.09} & 0.12 & 772.66 & \textbf{97.80} & 146.61 & 0.15 & \textbf{0.10} & 0.12 & 595.05 & 137.26 & \textbf{122.53} \\
	200 & 0.13 & \textbf{0.09} & 0.10 & 1414.81 & 447.01 & \textbf{337.59} & 0.13 & \textbf{0.08} & 0.10 & 1501.27 & \textbf{309.48} & 370.49 \\
	500 & 0.09 & \textbf{0.07} & 0.08 & 2858.86 & \textbf{1306.55} & 1503.84 & 0.09 & \textbf{0.06} & 0.09 & 3778.31 & \textbf{714.32} & 1715.89 \\
\end{tabular}
\end{table}

\subsection{Case Study} \label{sec:real_data}

In the following, we examine the estimators by analyzing real data. On the one hand we consider eye movements, as recorded via electroencephalography (EEG), and on the other hand videos.

\subsubsection{EEG-Data Analysis}

We used the ``Consumer-Grade EEG and Eye-Tracking Dataset'', which contains 116 EEG recordings with simultaneously measured eye movements \citep{afonso2025}. The corresponding challenge is the reconstruction of eye movements from the EEG. For the following experiments, we considered the ``level-2-smooth'' recordings, where the participants' eyes followed a smoothly moving target on a screen. Only recordings without technical errors have been considered.
The EEG part of the dataset consists of 4 EEG signals recorded at different positions on the scalp with a sampling rate of 256 Hz. We first transformed the four dimensional time series $(Y_i)_{i=1}^n$ to a (discretized) multivariate functional time series by defining 
\[ X_i(\tfrac{j}{m-1}) = Y_{5i + j} \]
for $j=0, \dots, m-1, m=50$ and $i=1, \dots, \lfloor \tfrac{n}{5}\rfloor - (m-1)=:\nu$. With this definition, each observation $X_i$ can be considered as an element of $\big(L^2([0, 1])\big)^4$, for $i=1, \dots, \nu$. As target variable, we considered the $x$- and $y$-coordinate of the moving target on the screen. 

In addition to the three estimators $\tilde{\mu}_n, \hat{\mu}_{h_n}$ and $\hat{\mu}_{NW}$, we use the originally proposed frequency filter to smooth the data.
To compare the different methods, we first smoothed the time series $(X_i)_{i=1}^\nu$. For the kernel estimators, we used the quartic kernel, as before, and selected $h_n = \tfrac{50}{\nu}$. Subsequently, we computed the $L^2$-norm of the smoothed time series per coordinate, yielding a four-dimensional time series. Then, we determined the offset between the EEG and the target on screen by finding the lag, that maximizes the cross-correlation function between these time series. Finally, we fitted a linear model, that takes the time-corrected, smooth EEG signal at some time point $i$ as input and predicts the $x$- and $y$-coordinate of the target at the same time point. To evaluate the quality of the reconstructed eye movements, we calculated the correlation $\rho_x$ and $\rho_y$ between true and predicted $x$- and $y$-coordinate, respectively. The results are reported in Table \ref{tab:eeg} and the reconstructed signal for recording \texttt{P085\_01} is visualized in Figure \ref{fig:eeg}.

In line with previous results, the estimates of the Nadaraya-Watson and the local linear estimator are very similar. The Jackknife-based predictions yield slightly lower correlations, and the predictions based on the frequency filter are significantly worse. The $x$-coordinate seems easier to reconstruct, which is likely due to the electrodes' positions. The electrodes are spread horizontally across the forehead and can therefore measure lateral movements better. The standard deviation is generally high, which suggests that the reconstruction is easier for some participants than for others.

\begin{table}
	\caption{Average correlation between true and predicted coordinates of the target, and average processing time (in ms), across all recordings (standard deviation in parentheses).}	
	\label{tab:eeg}
	\begin{tabular}{r|rrrr}
		& NWE & LLE & Jackknife & Filter \\
		\midrule
		$\rho_x$ & \textbf{0.47} ($\pm$ 0.16) & \textbf{0.47} ($\pm$ 0.16) & 0.41 ($\pm$ 0.15) & 0.12 ($\pm$ 0.05) \\
		$\rho_y$ & \textbf{0.37} ($\pm$ 0.16) & \textbf{0.37} ($\pm$ 0.16) & 0.29 ($\pm$ 0.12) & 0.13 ($\pm$ 0.06) \\
		Time (in ms)& \textbf{90.47} ($\pm$ 8.27) & 231.1 ($\pm$ 16.22) & 480.07 ($\pm$ 31.15) & 77.94 ($\pm$ 6.65) \\
	\end{tabular}
\end{table}

\begin{figure}
	\centering
	\includegraphics[width=0.9\textwidth]{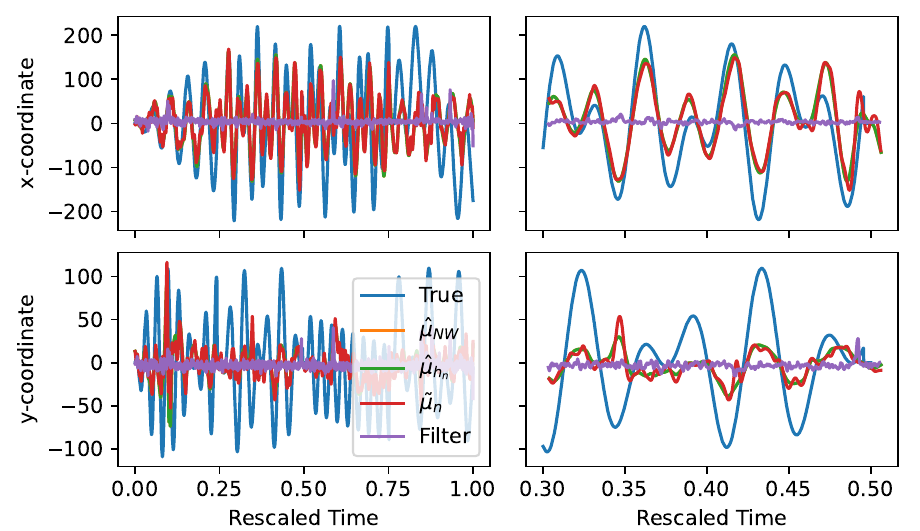}
	\caption{Reconstruction of target position based on EEG data, where the data was smoothed with different methods.}
	\label{fig:eeg}
\end{figure}

\subsubsection{Video Analysis}

Typical tasks in the analysis of (functional) time series, are the detection of outliers and change points. For this purpose, we used two videos of the CAVIAR dataset \citep{fisher2005}. Videos can be naturally regarded as functional time series, where each observation $X_t$ is a function of an $x$ and $y$ coordinate.

We first considered the video ``\texttt{WalkByShop1front}''. The task is to detect pedestrians passing through the video as ``outliers'' compared to an otherwise constant background. We used the Jackknife, local linear and Nadaraya-Watson estimators to smooth the video, with a bandwidth of $50$ frames, which corresponds to $2$\,s. Subsequently, we calculated the residuals as difference between the original video and its smoothed version, and computed the $L^2$-norm of the residuals, for each time point $t$. The resulting univariate time series is displayed in Figure \ref{fig:shop}. Time points, where pedestrians pass through the video can be easily identified by the peaks. 

\begin{figure}
	\centering
	\includegraphics[width=0.9\textwidth]{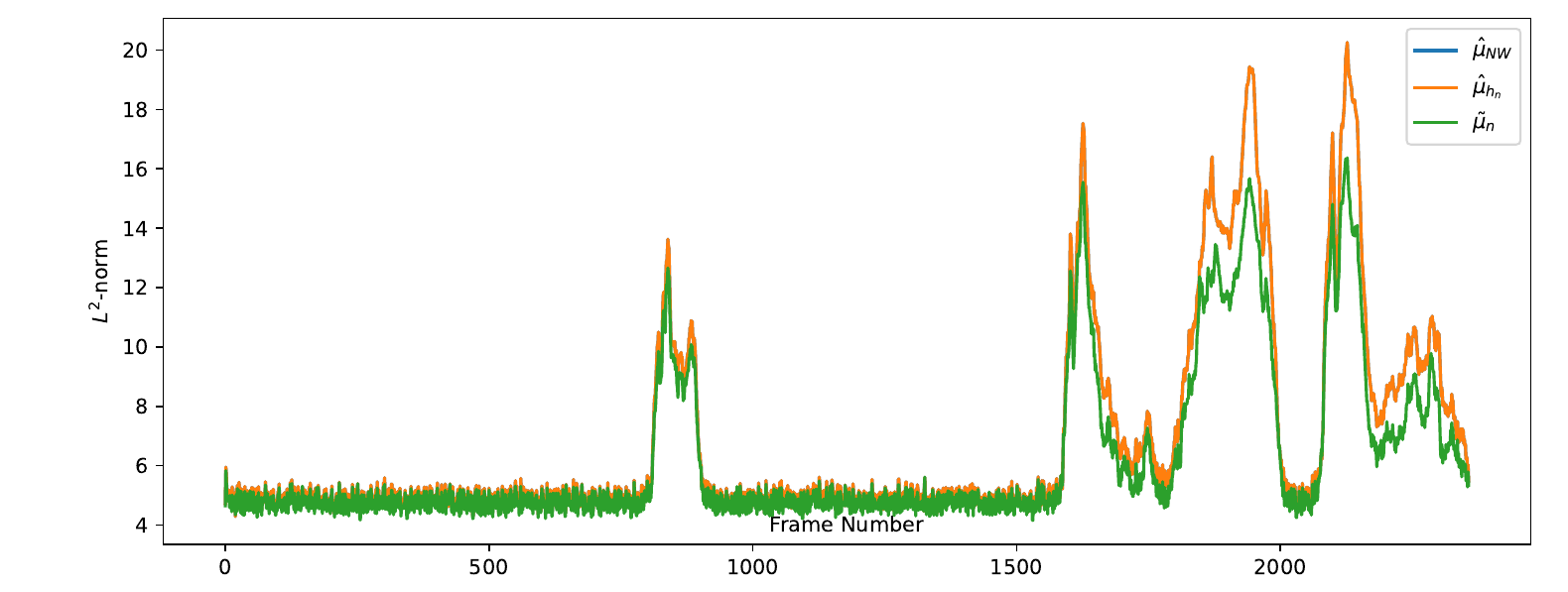}
	\caption{Sup norm of the residuals for various estimators.}
	\label{fig:shop}
\end{figure}

Second, we considered the video ``\texttt{LeftBag\_AtChair}'' between frames $254$ and $744$, which corresponds to $19.6$\,s. Here the task is to detect the time at which a bag is left on the floor by a pedestrian. We smoothed the video with the different estimators and a bandwidth of $100$ frames, and calculated residuals as before. Next, we calculated the sup norm of each residual, and finally the CUSUM process of the resulting univariate time series. The CUSUM processes are displayed in Figure \ref{fig:bag}. The CUSUM process reaches its maximum in frame $283$ (frame $537 = 283 + 254$ in the full video), for all estimators, which corresponds to the time point, in which the bag is dropped by the pedestrian.

The smoothed videos and the residuals are available online: \url{https://github.com/FlorianHeinrichs/banach_space_llr/tree/main/videos}

\begin{figure}
	\centering
	\includegraphics[width=0.9\textwidth]{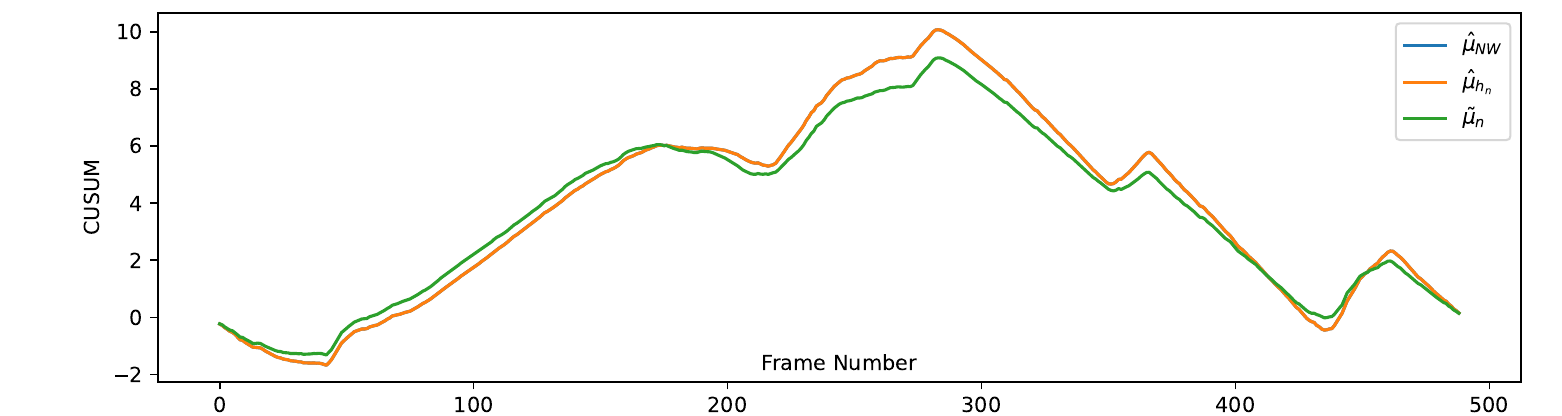}
	\caption{CUSUM process of the residuals' sup norm for various estimators.}
	\label{fig:bag}
\end{figure}

\section{Proof of Theorem \ref{thm:lle}} \label{sec:proof}

Before we prove the theorem, we state a sufficient condition for minima of convex functions.

\begin{proposition} \label{prop:convex_min}
	Let $V$ be a Banach space and $f:V\to\R$ a convex function. If $f$ is Gateaux differentiable in $x_0$ and $\diff f(x_0, \psi) = 0$ for any $\psi\in V$, then $x_0$ is a global minimum of $f$.
\end{proposition}

\begin{proof}
	Let $t\in (0, 1]$ and $x\in V$, then, by convexity of $f$,
	\[ f\big(x_0 + t(x-x_0)\big) \le f(x_0) + t\big(f(x)-f(x_0)\big). \]
	
	Rearranging terms yields
	\[ \frac{f\big(x_0 + t(x-x_0)\big) - f(x_0)}{t} \le f(x)-f(x_0). \]
	
	For $t\to 0$, the left-hand side converges to $\diff f (x_0, x-x_0)$, which is 0 by assumption. Thus, $f(x)\ge f(x_0)$ .
\end{proof}

\begin{proposition} \label{prop:taylor}
	Let $\mu$ satisfy  part \ref{assump:mu} of Assumption \ref{assump}. Then, the Taylor expansion 
	\[ \mu\big(\tfrac{i}{n}\big) = \mu(t) + D \mu(t) \big(\tfrac{i}{n}-t\big) + \tfrac{1}{2} D^2 \mu(t) \big(\tfrac{i}{n}-t\big)^2 + \Oc\Big(\big|\tfrac{i}{n}-t\big|^3\Big) \]
	holds uniformly in $t$.
\end{proposition}

\begin{proof}
	By Theorem 5.4 of \cite{coleman2012}, it holds
	\begin{align}\label{eq:int_remainder}
		\mu\big(\tfrac{i}{n}\big)  = &~ \mu(t) + D \mu(t) \big(\tfrac{i}{n}-t\big) + \int_0^1 (1-x) D^2 \mu\Big(t+ x \big(\tfrac{i}{n}-t\big)\Big) \big(\tfrac{i}{n}-t\big)^2 \diff x  \notag \\
		= &~ \mu(t) + D \mu(t) \big(\tfrac{i}{n}-t\big) + \tfrac{1}{2} D^2 \mu(t) \big(\tfrac{i}{n}-t\big)^2 \\
		& +  \int_0^1 (1-x) \big(\tfrac{i}{n}-t\big)^2 \Big[ D^2 \mu\Big(t+ x \big(\tfrac{i}{n}-t\big)\Big) - D^2 \mu(t)\Big] \diff x. \notag
	\end{align}  
	By Lipschitz continuity of $D^2 \mu$, 
	\[  \| D^2 \mu\Big(t+ x \big(\tfrac{i}{n}-t\big)\Big) - D^2 \mu(t) \| \le C \big|\tfrac{i}{n}-t\big|. \]
	Therefore, the integral on the right-hand side of \eqref{eq:int_remainder} can be bounded by $C \big|\tfrac{i}{n}-t\big|^3$, for some constant $C$ that does not depend on $\tfrac{i}{n}$ and $t$.
\end{proof}

\textbf{Proof of Theorem \ref{thm:lle}}~~ Define $\sigma(v)=\sup_{x\in[0, 1]} v(x)$ for $v\in \C$. By Corollary 2.3 of \cite{carcamo2020}, the Gateaux-derivative of $\sigma$ is given by $\diff \sigma(v; \psi) = \sup_{x\in \Ec(v)} \psi(x) = \sup_{\Ec(v)} \psi$, for
\[ \Ec(v) = \{ x\in[0, 1]: v(x) = \sup_{y\in[0, 1]} v(y) \}. \] 

Further, define
\[ g(b) = \frac{1}{nh_n}\sum_{i=1}^{n} \Big(X_{i} - b_0 - b_1 \big(\tfrac{i}{n}-t\big) \Big)^2 K\big(\tfrac{i-nt}{nh_n}\big). \]

Then, to calculate $\hat{\mu}_{h_n}(t)$, we need to minimize $f(b) = (\sigma \circ g) (b)$. 
By the chain rule,  $\diff f(b; \psi) = \diff \sigma\big(g(b); \diff g(b; \psi)\big)$. 
By straightforward calculations, it follows
\[ \diff g(b;\psi) = - 2 \frac{1}{nh_n}\sum_{i=1}^{n} \Big(X_{i} - b_0 - b_1 \big(\tfrac{i}{n}-t\big) \Big) \Big(\psi_0 + \psi_1 \big(\tfrac{i}{n}-t\big) \Big) K\big(\tfrac{i-nt}{nh_n}\big). \]

Thus, we obtain the derivative
\begin{align*}
	\diff f(b; \psi) & = \sup_{\Ec(g(b))} - 2 \frac{1}{nh_n} \sum_{i=1}^{n} \Big(X_{i} - b_0 - b_1 \big(\tfrac{i}{n}-t\big) \Big) \Big(\psi_0 + \psi_1 \big(\tfrac{i}{n}-t\big) \Big) K\big(\tfrac{i-nt}{nh_n}\big)\\
	& = \sup_{\Ec(g(b))} - 2 \psi_0 \big[R_0(t)-b_0 S_0(t) - b_1  h_n S_1(t) \big]
	- 2 \psi_1 h_n \big[R_1(t)-b_0 S_1(t) - b_1 h_n S_2(t) \big],
\end{align*}

where 
\begin{equation}\label{eq:s_r}
	S_\ell(t) = \frac{1}{nh_n} \sum_{i=1}^{n} (\tfrac{i-nt}{nh_n}\big)^\ell K\big(\tfrac{i-nt}{nh_n}\big)
	\quad\text{and}\quad
	R_\ell(t) = \frac{1}{nh_n} \sum_{i=1}^{n} X_i (\tfrac{i-nt}{nh_n}\big)^\ell K\big(\tfrac{i-nt}{nh_n}\big),
\end{equation}
for $\ell \in \{0, 1, 2\}$ and $t\in[0, 1]$. By convexity of $f$ and Proposition \ref{prop:convex_min}, any $b\in \C^2$ with $\diff f(b; \psi)=0$ for any $\psi\in \C^2$ is a global minimum of $f$.

Note that $S_\ell(t)\in\R$ and $R_\ell(t)\in \C$, for $t\in[0, 1]$. Further note, that $\diff f(b; \psi) = 0$ for any $\psi\in \C^2$, if 
\[ R_\ell(t)-b_0 S_\ell(t) - b_1  h_n S_{\ell+1}(t) = 0 \]
for $\ell\in\{0, 1\}$, which can be rewritten as 
\begin{align*}
	\left(\begin{array}{c} b_0 \\ b_1 \end{array}\right)
	= \left(\begin{array}{cc}
		S_0(t) & h_n S_1(t)\\
		S_1(t) & h_n S_2(t)
	\end{array}\right)^{-1}  \left(\begin{array}{c}
		R_0(t)\\
		R_1(t)\\
	\end{array}\right).
\end{align*} 

By using the Taylor expansion from Proposition \ref{prop:taylor}, we obtain
\begin{align}\notag 
	\sup_{t\in I_n} \bigg\| & \hat{\mu}_{h_n}(t) - \mu(t) - h_n \tfrac{S_2(t) S_1(t)-S_1^2(t)}{S_0(t)S_2(t)-S_1^2(t)} D\mu(t) -\tfrac{h_n^2}{2}  \tfrac{S_2^2(t) -S_1(t)S_3(t)}{S_0(t)S_2(t)-S_1^2(t)} D^2 \mu(t)  \\ 
	& - \frac{1}{nh_n} \sum_{i=1}^n \tfrac{S_2(t)-S_1(t)(\tfrac{i-nt}{nh_n})}{S_0(t)S_2(t)-S_1^2(t)}  K\big(\tfrac{i-nt}{nh_n}\big) \eps_i \bigg\| = \Oc(h_n^3). \label{eq:taylor_expansion}
\end{align}

By Lipschitz continuity of $K$, it holds $\sup_{t\in I_n}|S_\ell(t) - \int_{-1}^1 y^\ell K(y)\diff y | = \Oc(\tfrac{1}{nh_n})$. In particular, $S_0(t)=1 + \Oc(\tfrac{1}{nh_n})$, and by symmetry of $K$, $S_\ell(t) = \Oc(\tfrac{1}{nh_n})$, for $\ell=1, 3$. Further, recall that $\kappa_2 = \int_{-1}^1 y^2 K(y) \diff y$. With this notation, \eqref{eq:taylor_expansion} simplifies to 
\begin{equation} \label{eq:taylor_expansion2}
	\sup_{t\in I_n} \bigg\| \hat{\mu}_{h_n}(t) - \mu(t) - \tfrac{h_n^2}{2} \kappa_2 D^2 \mu(t)  \\ 
	- \frac{1}{nh_n} \sum_{i=1}^n \tfrac{S_2(t)-S_1(t)(\tfrac{i-nt}{nh_n})}{S_0(t)S_2(t)-S_1^2(t)}  K\big(\tfrac{i-nt}{nh_n}\big) \eps_i \bigg\| = \Oc(h_n^3+\tfrac{1}{n}).
\end{equation}

Note that $\sum_{i=1}^n K^\ell\big(\tfrac{i-nt}{nh_n}\big)$ contains only $2nh_n+1$ non-zero summands, so that it is of order $\Oc(nh_n)$, uniformly for in $t \in I_n$ and $\ell\in\R$.
Further, 
\[\sup_{t\in I_n} \bigg|\frac{S_2(t)-S_1(t)(\tfrac{i-nt}{nh_n})}{S_0(t)S_2(t)-S_1^2(t)} - 1\bigg| =  \Oc(\tfrac{1}{nh_n}).\]
Thus, for $p >1$, it follows from the triangle and Hölder's inequality
\begin{align*}
	& \ex\bigg[ \sup_{t\in I_n}  \bigg\| \frac{1}{nh_n} \sum_{i=1}^n \Big( \tfrac{S_2(t)-S_1(t)(\tfrac{i-nt}{nh_n})}{S_0(t)S_2(t)-S_1^2(t)} - 1 \Big)  K\big(\tfrac{i-nt}{nh_n}\big) \eps_i  \bigg\|^p \bigg] \\
	& \le \ex\bigg[ \sup_{t\in I_n} \bigg(\frac{1}{nh_n} \sum_{i=1}^n \Big| \tfrac{S_2(t)-S_1(t)(\tfrac{i-nt}{nh_n})}{S_0(t)S_2(t)-S_1^2(t)} - 1 \Big| K\big(\tfrac{i-nt}{nh_n}\big)  \| \eps_i \|\bigg)^p \bigg] \\
	& \le\frac{1}{(nh_n)^p} \sum_{i=1}^n  \ex[\| \eps_i \|^p] \sup_{t\in I_n}  \bigg( \sum_{i=1}^n \Big| \tfrac{S_2(t)-S_1(t)(\tfrac{i-nt}{nh_n})}{S_0(t)S_2(t)-S_1^2(t)} - 1 \Big|^{\frac{p}{p-1}}
	 \Big(K\big(\tfrac{i-nt}{nh_n}\big)\Big)^{\frac{p}{p-1}} \bigg)^{p-1}  
	 = \Oc\Big(\tfrac{1}{n^p h_n^{p+1}}\Big),
\end{align*}
where we used part \ref{assump:errors} of Assumption \ref{assump} in the last step. For $p=1$, we can use the same arguments while replacing the last supremum by 
\[ \sup_{t\in I_n} \max_{i=1}^n \Big| \tfrac{S_2(t)-S_1(t)(\tfrac{i-nt}{nh_n})}{S_0(t)S_2(t)-S_1^2(t)} - 1 \Big|
K\big(\tfrac{i-nt}{nh_n}\big) = \Oc(\tfrac{1}{nh_n}). \]

In particular, \eqref{eq:taylor_expansion2} further simplifies to 
\begin{equation*} 
	\sup_{t\in I_n} \bigg\| \hat{\mu}_{h_n}(t) - \mu(t) - \tfrac{h_n^2}{2} \kappa_2 D^2 \mu(t)   
	- \frac{1}{nh_n} \sum_{i=1}^n  K\big(\tfrac{i-nt}{nh_n}\big) \eps_i \bigg\| = \Oc(h_n^3) + \Oc_\pr\Big(\tfrac{1}{nh_n^{1+1/p}}\Big).
\end{equation*}

By the same arguments, it follows
\begin{equation*} 
	\sup_{t\in I_n} \bigg\| h_n\big(\widehat{D\mu}_{h_n}(t) - D\mu(t)\big) - \tfrac{h_n^2}{2} \frac{\kappa_3}{\kappa_2} D^2 \mu(t)  \\ 
	- \frac{1}{nh_n} \sum_{i=1}^n \frac{1}{\kappa_2}\big(\tfrac{i-nt}{nh_n}\big) K\big(\tfrac{i-nt}{nh_n}\big) \eps_i \bigg\| = \Oc(h_n^3) + \Oc_\pr\Big(\tfrac{1}{nh_n^{1+1/p}}\Big).
\end{equation*}

The theorem follows from the definition of the Jackknife estimators. \hfill$\blacksquare$

\bibliography{bibliography}

\newpage

\appendix


\section{Additional Empirical Results} \label{app:appendix_empirical_results}

\begin{table}[h]
	\caption{Mean processing times of the compared estimators for different choices of $\mu$ and $\eps$ in ms.}
	\label{tab:times}
	\begin{tabular}{l|lll|lll|lll|lll}
		& \multicolumn{3}{c|}{50} & \multicolumn{3}{c|}{100} & \multicolumn{3}{c|}{200} & \multicolumn{3}{c}{500} \\
		$\eps$ & $\tilde{\mu}_n$ & $\hat{\mu}_{h_n}$ & $\hat{\mu}_{NW}$ & $\tilde{\mu}_n$ & $\hat{\mu}_{h_n}$ & $\hat{\mu}_{NW}$ & $\tilde{\mu}_n$ & $\hat{\mu}_{h_n}$ & $\hat{\mu}_{NW}$ & $\tilde{\mu}_n$ & $\hat{\mu}_{h_n}$ & $\hat{\mu}_{NW}$ \\
		\midrule
		\addlinespace[.2cm]
		\multicolumn{13}{l}{\quad\textit{Panel A: $\mu_1$ }} \\
		(BB) & 0.86 & 0.43 & 0.21 & 1.14 & 0.58 & 0.27 & 1.83 & 0.98 & 0.42 & 3.41 & 1.83 & 0.79 \\
		(BM) & 0.85 & 0.44 & 0.21 & 1.14 & 0.56 & 0.26 & 1.68 & 0.83 & 0.36 & 3.47 & 1.74 & 0.76 \\
		(FAR-BB) & 0.86 & 0.43 & 0.21 & 1.16 & 0.57 & 0.27 & 1.65 & 0.81 & 0.36 & 3.45 & 1.74 & 0.74 \\
		(FAR-BM) & 0.87 & 0.43 & 0.20 & 1.15 & 0.56 & 0.26 & 1.75 & 0.87 & 0.37 & 3.38 & 1.67 & 0.71 \\
		(tvBM) & 0.87 & 0.43 & 0.20 & 1.18 & 0.59 & 0.27 & 1.66 & 0.95 & 0.41 & 3.40 & 1.75 & 0.75 \\
		(tvFAR1) & 0.89 & 0.43 & 0.20 & 1.10 & 0.57 & 0.26 & 1.65 & 0.89 & 0.37 & 3.37 & 1.63 & 0.71 \\
		(tvFAR2) & 0.87 & 0.42 & 0.20 & 1.12 & 0.55 & 0.25 & 1.68 & 0.82 & 0.35 & 3.34 & 1.66 & 0.71 \\
		\addlinespace[.2cm]
		\multicolumn{13}{l}{\quad\textit{Panel B: $\mu_2$ }} \\
		(BB) & 0.87 & 0.43 & 0.21 & 1.13 & 0.57 & 0.26 & 1.66 & 0.81 & 0.37 & 3.49 & 1.74 & 0.75 \\
		(BM) & 0.87 & 0.43 & 0.21 & 1.12 & 0.56 & 0.26 & 1.69 & 0.82 & 0.36 & 3.53 & 1.74 & 0.76 \\
		(FAR-BB) & 0.87 & 0.43 & 0.20 & 1.15 & 0.57 & 0.26 & 1.65 & 0.81 & 0.36 & 3.60 & 1.71 & 0.73 \\
		(FAR-BM) & 0.87 & 0.43 & 0.20 & 1.14 & 0.55 & 0.25 & 1.65 & 0.80 & 0.35 & 3.38 & 1.68 & 0.73 \\
		(tvBM) & 0.88 & 0.43 & 0.21 & 1.15 & 0.58 & 0.27 & 1.64 & 0.80 & 0.36 & 3.47 & 1.83 & 0.76 \\
		(tvFAR1) & 0.85 & 0.43 & 0.20 & 1.13 & 0.56 & 0.26 & 1.64 & 0.81 & 0.35 & 3.39 & 1.70 & 0.71 \\
		(tvFAR2) & 0.88 & 0.43 & 0.20 & 1.14 & 0.55 & 0.25 & 1.64 & 0.81 & 0.36 & 3.45 & 1.67 & 0.72 \\
	\end{tabular}
\end{table}

\begin{table}
	\caption{MAE of the compared estimators for different choices of $\mu$ and $\eps$.}
	\label{tab:results2}
	\begin{tabular}{r|rrrrrr|rrrrrr}
		& \multicolumn{6}{c|}{$\mu_1$} & \multicolumn{6}{c}{$\mu_2$} \\
		$n$ & $\tilde{\mu}_n$ & $\hat{\mu}_{h_n}$ & $\hat{\mu}_{NW}$ & $\widetilde{D\mu}_n$ & $\widehat{D\mu}_{h_n}$ & $\widehat{D\mu}_{NW}$ & $\tilde{\mu}_n$ & $\hat{\mu}_{h_n}$ & $\hat{\mu}_{NW}$ & $\widetilde{D\mu}_n$ & $\widehat{D\mu}_{h_n}$ & $\widehat{D\mu}_{NW}$ \\
		\midrule
		\addlinespace[.2cm]
		\multicolumn{13}{l}{\quad\textit{Panel A: (BM)}} \\
		50 & 0.35 & \textbf{0.19} & \textbf{0.19} & 26.57 & 3.20 & \textbf{2.61} & 0.29 & \textbf{0.19} & 0.20 & 15.09 & 3.35 & \textbf{3.08} \\
		100 & 0.21 & \textbf{0.15} & 0.16 & 12.42 & 3.59 & \textbf{3.04} & 0.21 & \textbf{0.16} & \textbf{0.16} & 12.29 & 3.64 & \textbf{3.19} \\
		200 & 0.19 & \textbf{0.13} & \textbf{0.13} & 17.10 & 4.12 & \textbf{3.43} & 0.19 & \textbf{0.13} & \textbf{0.13} & 17.08 & 4.18 & \textbf{3.60} \\
		500 & 0.15 & \textbf{0.10} & \textbf{0.10} & 20.21 & 5.04 & \textbf{4.20} & 0.15 & \textbf{0.10} & \textbf{0.10} & 20.15 & 5.07 & \textbf{4.22} \\
		\addlinespace[.2cm]
		\multicolumn{13}{l}{\quad\textit{Panel B: (BB)}} \\
		50 & 0.18 & \textbf{0.11} & \textbf{0.11} & 11.49 & 1.73 & \textbf{1.42} & 0.17 & \textbf{0.12} & \textbf{0.12} & 8.71 & 2.24 & \textbf{2.15} \\
		100 & 0.12 & \textbf{0.09} & \textbf{0.09} & 7.18 & 1.96 & \textbf{1.64} & 0.13 & \textbf{0.10} & \textbf{0.10} & 7.18 & 2.23 & \textbf{2.02} \\
		200 & 0.11 & \textbf{0.08} & \textbf{0.08} & 9.93 & 2.39 & \textbf{1.97} & 0.11 & \textbf{0.08} & \textbf{0.08} & 9.90 & 2.41 & \textbf{2.06} \\
		500 & 0.09 & \textbf{0.06} & \textbf{0.06} & 11.79 & 2.92 & \textbf{2.43} & 0.09 & \textbf{0.06} & \textbf{0.06} & 11.72 & 2.93 & \textbf{2.44} \\
		\addlinespace[.2cm]
		\multicolumn{13}{l}{\quad\textit{Panel C: (FAR-BM)}} \\
		50 & 0.35 & \textbf{0.25} & 0.27 & 16.49 & 4.96 & \textbf{4.53} & 0.36 & \textbf{0.25} & 0.27 & 16.72 & 4.61 & \textbf{4.52} \\
		100 & 0.28 & \textbf{0.22} & 0.24 & 16.77 & \textbf{6.27} & 6.66 & 0.28 & \textbf{0.23} & 0.25 & 16.61 & 7.52 & \textbf{7.21} \\
		200 & 0.26 & \textbf{0.21} & 0.22 & 25.05 & 11.00 & \textbf{10.59} & 0.26 & \textbf{0.21} & 0.22 & 24.59 & 12.28 & \textbf{11.61} \\
		500 & 0.21 & \textbf{0.19} & \textbf{0.19} & 36.10 & 24.50 & \textbf{20.07} & 0.22 & \textbf{0.19} & 0.20 & 39.12 & 24.98 & \textbf{21.87} \\
		\addlinespace[.2cm]
		\multicolumn{13}{l}{\quad\textit{Panel D: (FAR-BB)}} \\
		50 & 0.19 & \textbf{0.13} & \textbf{0.13} & 9.31 & 2.10 & \textbf{1.86} & 0.19 & \textbf{0.14} & \textbf{0.14} & 9.30 & \textbf{2.42} & 2.53 \\
		100 & 0.15 & \textbf{0.11} & \textbf{0.11} & 8.25 & 2.85 & \textbf{2.24} & 0.15 & \textbf{0.12} & \textbf{0.12} & 8.23 & 2.96 & \textbf{2.88} \\
		200 & 0.13 & \textbf{0.09} & 0.10 & 11.92 & 3.32 & \textbf{3.15} & 0.13 & \textbf{0.10} & \textbf{0.10} & 11.69 & 3.65 & \textbf{3.21} \\
		500 & 0.11 & \textbf{0.07} & \textbf{0.07} & 14.77 & 4.26 & \textbf{3.64} & 0.11 & \textbf{0.08} & \textbf{0.08} & 14.88 & 4.40 & \textbf{3.89} \\
		\addlinespace[.2cm]
		\multicolumn{13}{l}{\quad\textit{Panel E: (tvBM)}} \\
		50 & 0.29 & \textbf{0.19} & \textbf{0.19} & 14.87 & 3.40 & \textbf{2.68} & 0.30 & \textbf{0.20} & \textbf{0.20} & 15.66 & 3.66 & \textbf{3.12} \\
		100 & 0.22 & \textbf{0.16} & \textbf{0.16} & 12.85 & 4.21 & \textbf{3.09} & 0.21 & \textbf{0.16} & \textbf{0.16} & 12.34 & 3.83 & \textbf{3.29} \\
		200 & 0.19 & \textbf{0.13} & \textbf{0.13} & 17.48 & 4.23 & \textbf{3.48} & 0.19 & \textbf{0.13} & \textbf{0.13} & 17.21 & 4.35 & \textbf{3.63} \\
		500 & 0.15 & \textbf{0.10} & \textbf{0.10} & 20.40 & 5.13 & \textbf{4.24} & 0.15 & \textbf{0.10} & \textbf{0.10} & 20.22 & 5.13 & \textbf{4.28} \\
		\addlinespace[.2cm]
		\multicolumn{13}{l}{\quad\textit{Panel F: (tvFAR1)}} \\
		50 & 0.41 & \textbf{0.25} & 0.28 & 19.95 & 5.33 & \textbf{5.07} & 0.36 & \textbf{0.25} & 0.27 & 16.68 & \textbf{4.68} & 4.82 \\
		100 & 0.28 & \textbf{0.23} & 0.25 & 16.25 & \textbf{7.77} & 8.00 & 0.29 & \textbf{0.25} & \textbf{0.25} & 16.96 & 9.31 & \textbf{7.32} \\
		200 & 0.26 & \textbf{0.20} & 0.23 & 26.36 & \textbf{10.74} & 12.17 & 0.26 & \textbf{0.20} & 0.22 & 27.74 & \textbf{10.64} & 11.54 \\
		500 & 0.22 & \textbf{0.19} & 0.20 & 39.10 & \textbf{22.22} & 23.35 & 0.22 & \textbf{0.19} & 0.20 & 38.98 & 25.03 & \textbf{21.39} \\
		\addlinespace[.2cm]
		\multicolumn{13}{l}{\quad\textit{Panel G: (tvFAR2)}} \\
		50 & 0.36 & \textbf{0.26} & 0.27 & 17.10 & 5.16 & \textbf{4.55} & 0.36 & \textbf{0.25} & 0.27 & 16.44 & \textbf{4.16} & 4.74 \\
		100 & 0.30 & \textbf{0.22} & 0.25 & 18.96 & \textbf{6.23} & 7.40 & 0.29 & \textbf{0.23} & 0.24 & 17.16 & 7.54 & \textbf{7.02} \\
		200 & 0.26 & \textbf{0.22} & \textbf{0.22} & 26.09 & 12.99 & \textbf{11.11} & 0.26 & \textbf{0.21} & 0.23 & 26.67 & \textbf{11.01} & 11.75 \\
		500 & 0.21 & \textbf{0.19} & 0.20 & 36.03 & \textbf{21.31} & 22.53 & 0.22 & \textbf{0.17} & 0.21 & 40.62 & \textbf{16.39} & 24.67 \\
	\end{tabular}
\end{table}

\end{document}